\documentclass[12pt]{article}
\usepackage{setspace}
\usepackage{amsmath,amssymb,amsfonts,amsthm}
\newcounter{item}[section]
\newcounter{kirshr}
\newcounter{kirsha}
\newcounter{kirshb}
\newtheorem{theorem}{Theorem}[section]

\theoremstyle{definition}

\newtheorem{definition}[theorem]{Definition}

\def\A{{\mathfrak{A}}}
\def\B{{\mathfrak{B}}}

\def\(R)RA{{\bf (R)RA}}

\def\B{{\sf B}}

\def\A{{\mathfrak{A}}}
\def\B{{\mathfrak{B}}}

\def\A{{\cal{A}}}
\def\B{{\mathfrak{B}}}

\title{On substitution algebras of permutations}
\author{Mohammad Assem}
\begin{document}
\maketitle
\begin{abstract}
The subject of this paper is a simulation to that in \cite{sagiphd} but here we consider substitutions corresponding to transpositions instead of replacements.
\end{abstract}
\section{Transposition Algebras}

\begin{definition}[Transposition Set Algebras]
Let $U$ be a set. \emph{The full transposition set algebra of dimension} $\alpha$ \emph{with base}
$U$ is the algebra $$\langle\mathcal{P}({}^\alpha U); \cap,\sim,S_{ij}\rangle_{i\neq j\in\alpha};$$
where $\sim$ means complement w.r.t. ${}^\alpha U$  and $S_{ij}$'s are unary operations defined
by $$S_{ij}(X)=\{q\in {}^\alpha U:q\circ [i,j]\in X\}.$$ Recall that $[i,j]$ denotes that transposition of $\alpha$ that permutes $i,j$
and leaves any other element fixed. The class of {\it Transposition Set Algebras of dimension $\alpha$} is defined as follows:
$$SetTA_\alpha=\mathbf{S}\{\A:\A\text{ is a full transposition set algebra of dimension }\alpha $$$$ \text{ with base }U,\text{ for some set }U\}.$$
\end{definition}

\begin{definition}[\emph{Representable Transposition Set Algebras}]
The class of {\it Representable Transposition Set Algebras} of dimension $\alpha$ is defined to be $$RTA_\alpha=\mathbf{SP}SetTA_\alpha.$$
\end{definition}

\begin{definition}[\emph{Permutable Set}]
Let $U$ be a given set, and let $D\subset{}^\alpha U.$ We say that $D$ is \emph{permutable} iff
$$(\forall i\neq j\in\alpha)(\forall s\in{}^\alpha U)(s\in D\Longrightarrow s\circ [i,j]\in D).$$
\end{definition}

\begin{definition}[\emph{Permutable Algebras}]
The class of \emph{Permutable Algebras} of dimension $\alpha$, $\alpha$ an ordinal,  is defined to be
$$PTA_{\alpha}=\mathbf{SP}\{\langle\mathcal{P}(D); \cap,\sim,S_{ij}\rangle_{i\neq j\in \alpha}: U\text{ \emph{is a set}}, D\subset{}^{\alpha}
U\text{\emph{permutable}}\}.$$
Here  $S_{ij}(X)=\{q\in D:q\circ [i,j]\in X\}$, and $\sim$ is complement w.r.t. $D$.\\
If $D$ is a permutable set, then we denote the algebra
$\wp(D)=\langle\mathcal{P}(D);\cap,\sim,S_{ij}\rangle_{i\neq j\in\alpha}$ by $\wp(D)$.
\end{definition}
Note that $\wp(^{\alpha}U)$ can be viewed as the complex algebra of the atom structure $(^{\alpha}U, S_{ij})_{i,j\in \alpha}$ where
for all $i,j, S_{ij}$ is a binary relation, such that for $s,t\in {}^{\alpha}U$, $(s,t)\in S_{ij}$ iff $s\circ [i,j]=t.$ When we consider permutable sets
then from the modal point of view we are restricting or relativizing the states or assignments to $D$.

For some time to come, we restrict ourselves to finite $\alpha\geq 2$, which we denote by $n$. Now we distinguish elements of $SetTA_n$ which play an important role.
\begin{definition}[\emph{Small algebras}]
For any natural number $k\leq n$
the algebra $\A_{nk}$ is defined to be $$\A_{nk}=\langle\mathcal{P}({}^nk);\cap,\sim,S_{ij}\rangle_{i\neq j\in n}.$$ So $\A_{nk}\in SetTA_n$.
\end{definition}

We will see that $\{\A_{nk}:k\leq n\}$ generates $RTA_n.$ Now we turn to showing that the class $RTA_n$ is not a variety.

\begin{theorem}\label{relativization}
Let $U$ be a set and suppose $G\subset{}^n U$ is permutable.
Let $\A=\langle\mathcal{P}({}^n U);\cap,\sim,S_{ij}\rangle_{i\neq j\in n}$ and let $\mathcal{B}=\langle\mathcal{P}(G);\cap,\sim ,S_{ij}\rangle_{i\neq j\in n}$.
Then the following map $h:\A\longrightarrow\mathcal{B}$ defined by $h(x)=x\cap G$ is a homomorphism.

\end{theorem}
\begin{proof}
It is easy to see that $h$ preserves $\cap,\sim$ so it remains to show that the $S_{ij}$'s are also preserved.
To do this let $i\neq j\in n$ and $x\in A.$ Now\begin{align*}
h(S_{ij}^\A x)= S_{ij}^\A x \cap G &=\{q\in{}^n U:q\circ [i,j]\in x\}\cap G\\
&=\{q\in G:q\circ [i,j]\in x\}\\
&=\{q\in G:q\circ [i,j]\in x\cap G\}\\
&=\{q\in G:q\circ [i,j]\in h(x) \}\\
&=S_{ij}^\mathcal{B} h(x)
\end{align*}
\end{proof}
The function $h$ will be called \emph{relativization by} $G$.

\begin{theorem}\label{small}
$RTA_n=\mathbf{SP}\{\A_{nk}:k\leq n\}.$
\end{theorem}
\begin{proof}
The proof is exactly like that of Theorem 4.9 in \cite{sagiphd}.
Clearly, $\{\A_{nk}:k\leq n\}\subset RTA_n,$ and since, by definition,
$RTA_n$ is closed under the formation of subalgebras and direct products, $RTA_n\supset\mathbf{SP}\{\A_{nk}:k\leq n\}.$

To prove the other inclusion, it is enough to show $SetTA_n\subset \mathbf{SP}\{\A_{nk}:k\leq n\}.$
Let $\A\in SetTA_n$ and suppose that $U$ is the base of $\A.$
If $U$ is empty, then $\A$ has one element, and one can easily show $\A\cong\A_{n0}.$
Otherwise for every $0^\A\neq a\in A$ we can construct a homomorphism $h_a$ such that
$h_a(a)\neq 0$ as follows. If $a\neq 0^\A$ then there is a sequence $q\in a.$ Let  $U_0^a=range(q)$.
Clearly, $^nU_0^a$ is permutable, therefore by Theorem \ref{relativization} relativizing by $^nU_0^a$ is a homomorphism to
$\A_{nk_a}$ (where $k_a:=|range(q)|\leq n$). Let $h_a$ be this homomorphism. Since $q\in {}^nU_0^a$ we have $h_a(a)\neq0^{\A_{nk_a}}.$
Applying Theorem \ref{folklore} one concludes that $\A\in\mathbf{SP}\{\A_{nk}:k\leq n\}$ as desired.
\end{proof}

Recall that a class $K$ of $BAO$'s is a quasi-variety if ${\bf SPUp}K=K$.
We now show:
\begin{theorem}For $n\geq 2$, $RTA_n$ is a quasi-variety
\end{theorem}
\begin{proof} It suffices to show that the ultraproduct of a system of set algebras is in $RTA_n$.
Let $I$ be a set, $F$ an ultrafilter on $I$, and let $(\A_i:i \in I)$ be a system of set algebras in $RTA_n$.
Assume that $\A_i\subset\wp(^nU_i)$. Let $U=\prod_{i\in I}U_i/F$, and define
$\psi:\prod \A_i/F\to \wp(^nU)$ by
$$(a_i:i\in I)/F\mapsto \{(s_0,\ldots, s_{n-1})/F: \{i\in I : (s_0(i),\ldots,s_{n-1}(i))\in a_i\}\in F\}.$$
The it can be easily checked that $\psi$ is an embedding.
\end{proof}

Now, is $RTA_n$ ($n\geq 2$) a variety ? The answer is No:

\begin{theorem}\label{not} For $n\geq 2$, $RTA_n$ is not a variety.
\end{theorem}
\begin{proof}
Let us denote by $\sigma$ the quasi-equation
$$s_fx+s_gx=-x\longrightarrow 0=1,$$ where $f,g$ can be any permutation.
We claim that for all $k\leq n,$ $\sigma$ holds in the small algebra $\A_{nk}$
(or more generally, any set algebra with square unit).
This can be seen using a constant map in $^nk.$ More precisely, let $q\in  {}^nk$
be an arbitrary constant map, and let $X$ be any subset of $^nk.$
We have two cases for $q$ which are $q\in X$ or $q\in \sim X$. In either case,
noticing that $q\in X\Leftrightarrow q\in S_f(X)\cup S_g(X),$ it cannot be $ S_f(X)\cup S_g(X)=\sim X.$
Thus, the implication $\sigma$ holds in $\A_{nk}.$
It follows then, from Theorem \ref{small},
that $RTA_n\models\sigma$ (because the operators $\mathbf{S}$ and $\mathbf{P}$ preserve quasi-equations).

Now we are going to show that there is some element $\B\in PTA_n$, and  specific permutations $f,g$, such that $\B\nvDash\sigma.$
Let $G\subset {}^nn$ be the following permutable set $$G=\{s\in {}^n2:|\{i:s(i)=1\}|=1\}.$$
Let $\B=\wp(G) \in PTA_n$ and take $f$ as the $n$-cycle $(0 \; 1\ldots n-1)$, and $g$ as the $n$-cycle $(n-1 \; n-2\ldots 0)$.
Let $X$ be the following subset of $G,$ $$X=\{e_i:i\mbox{ is odd, }i<n\},$$
where $e_i$ denotes the map that maps every element to $0$ except that the $i$th element is mapped to $1$.
It is easy to see that, for all $i<n,$ $e_{i+1}\circ f=e_i$ while $e_i\circ g=e_{i+1}$ with the exceptions that $e_0\circ f=e_{n-1}$ and $e_{n-1}\circ g=e_0.$
This clearly implies that $$S_f^\B(X)\cup S_g^\B(X)=\sim X=\{e_i:i\mbox{ is even, }i<n\}.$$
Since $0^\B\neq 1^\B,$ $X$ falsifies $\sigma$ in $\B.$ Since $\B\in {\bf H}\{{\wp(^nn)}\}$, we are done.
\end{proof}

\end{document}